\documentclass[a4paper]{article}

\usepackage[english]{babel}

\usepackage[dvips]{color} 
\usepackage{lscape}
\usepackage{amsmath,amssymb,amsthm}
\usepackage{amscd}
\usepackage{enumerate}
\usepackage{latexsym}
\usepackage{mathrsfs}
\usepackage[dvips]{graphicx}

\newcommand{\Q}{\mathbb{Q}}

\newcommand{\Z}{\mathbb{Z}}
\newcommand{\C}{\mathbb{C}}
\theoremstyle{definition}
\newtheorem{Def}{Definition}[section]
\newtheorem{Rem}[Def]{Remark}

\theoremstyle{theorem}
\newtheorem{Th}[Def]{Thorem}
\newtheorem{Prop}[Def]{Proposition}
\newtheorem{Lem}[Def]{Lemma}

\begin{document}
\title{
$\mathbb Q$-bases of the N\'eron-Severi  groups of certain elliptic surfaces
}
\author{Masamichi Kuroda}
\date{}
\maketitle
\begin{abstract}
P. Stiller computed the Picard numbers of several families of elliptic surfaces, 
the rank of the N\'eron-Severi groups
of these surfaces. 
However he did not give the generators of these groups. 
In this paper we give 
$\mathbb Q$-bases
of these groups explicitly. 
If these surfaces are rational, 
then we also show that they are $\Z$-bases. 
\end{abstract}

\section{Introduction}
\footnote[0]{{\hspace{-0.6cm} \it MSC2010:} 14J27 \\
{\it Keywords:} elliptic surfaces, N\'eron-Severi  groups}
An elliptic surface is a surface which has a surjective map onto a curve such that the generic fiber is a 
curve of genus one (see \cite{5}). 
The N\'eron-Severi group 
is the group of divisors modulo algebraic equivalence (see Exercise V.$1.7$ in \cite{7} ). 
This group is known to be a finitely generated abelian group, and
its rank is called the Picard number.  
P. Stiller computed the Picard numbers of several families of elliptic surfaces by studying the action 
of certain automorphisms 
on the cohomology group (\cite{1}). 
However he did not give the generators of these groups. 
The purpose of this paper is to give 
$\mathbb Q$-bases 
of the N\'eron-Severi groups of Stiller's list of elliptic surfaces ( Example $1, \ 2, \ 3, \ 4$ and $5$ in \cite{1}) explicitly. 

We now explain briefly how to construct such $\mathbb Q$-bases. 
Let 
${\cal E}$ be an elliptic surface. We denote by ${\rm NS} ({\cal E})$ the N\'eron-Severi group of ${\cal E}$. 
T. Shioda proved that the group ${\rm NS} ({\cal E})$ 
is generated by fibral divisors and horizontal divisors (\cite{2}). 
Here we mean by a fibral divisor a sum of irreducible components of 
fibers, 
and by a horizontal divisor a sum of images of sections. 
Let ${\cal E}_n \rightarrow {\mathbb P}_{\C}^1 \ (n \in {\mathbb N})$ be one 
of the families of elliptic surfaces 
of Example $1, \ 2, \ 3, \ 4$ and $5$ in \cite{1}, and let $E_n $ be the generic fiber of ${\cal E}_n$ for each $n$. 
$E_n \ (n \in {\mathbb N})$ are elliptic curves over the function field $\C (t)$. 
Computing the Picard number of an elliptic surface is equivalent to 
determining the Mordell-Weil rank (i.e., the rank of the Mordell-Weil group) of 
the generic fiber.  
In \cite{1}, he proved that 
for each family ${\cal E}_n \rightarrow {\mathbb P}_{\C}^1 \ (n \in {\mathbb N})$, 
there exists a finite set ${\rm Adm}_i \ (1 \leq i \leq 5)$ 
of natural numbers 
such that the Mordell-Weil rank of $E_n / \C (t)$ is 
$r = \sum_{d \mid n, \ d \in {\rm Adm}_i}^{} \varphi (d)$, 
where $\varphi $ is the Euler function, i.e., $\varphi (d)$ is the number of positive integers less than or equal to $d$ that are relatively prime to $d$. 
In this paper we shall construct $r$ rational points of $E_n$
in an ad hoc manner, and show the linear independence of the associated 
divisors in ${\rm NS} ({\cal E}_n)$. 
If ${\cal E}_n$ is rational, 
then we further show that they form a $\Z$-basis.

This paper is organized as follows. 
Section $2$ is a quick review of some basic 
results on 
the N\'eron-Severi groups of elliptic surfaces. 
Section 3 is the heart of this paper. We give a number of 
$\mathbb Q$-bases or $\Z$-bases of the N\'eron-Severi groups of Stiller's list of elliptic surfaces. 
In section $4$, 
we give an alternative proof of Stiller's computations of the Picard numbers. 

\section{The N\'eron-Severi  group of an elliptic surface}

In this paper, we mean by an elliptic surface a surjective morphism
$f : {\cal E} \rightarrow C$ onto a curve with a section (say zero section) 
such that the generic fiber of $f$ is an elliptic curve.
Let $f : {\cal E} \longrightarrow \mathbb{P}^1 _{\C}$ be a non-split minimal elliptic surface, 
and let $E / \C (t)$ be the generic fiber. 
There is a natural group isomorphism between the Mordell-Weil group of $E/ \C (t)$, denoted by $E (\C (t))$, and 
the group of sections of ${\cal E}$ over ${\mathbb P}_{\C} ^1$, denoted by 
${\cal E} ({\mathbb P}_{\C} ^1)$ (see Prop. $3.10.$ (c) in \cite{4}):
\begin{eqnarray} \label{isom}
E (\C (t)) \hspace{0.5cm}  &\overset{\sim }{\longrightarrow  }& \hspace{1.6cm} {\cal E} ({\mathbb P}_{\C} ^1), 
\end{eqnarray}
\vspace{-1cm}
\begin{eqnarray*}
\hspace{1.6cm} P = (x_P , \ y_P) \ &\longmapsto  &\  ( \sigma _P   : t \longmapsto (x_P (t) , \ y_P (t) , \ t)).
\end{eqnarray*}
According to the Mordell-Weil theorem, $E (\C (t))$ is a finitely generated group. 
In the following, for each $P \in E(\C (t))$, we denote by $(P)$ the image in ${\cal E}$ of the section corresponding to $P$. 
For simplicity, we denote by $\infty$ the image of zero section, that is, the section corresponding to 
zero element $O \in E (\C (t))$. 

The singular fibers are classified by Kodaira ({\rm see} \cite{5}). We shall follow Kodaira's notation. 
Let $\Sigma ({\cal E})$ be the finite set of points $t$ in ${\mathbb P}_{\C}^1$ 
such that ${\cal E}_t := f^{-1} (t)$ is a singular fiber. 
For each $t \in 
{\mathbb P}^1_{\C}
$, let $m_t$ be the number of irreducible components of the fiber ${\cal E}_t $, 
and we denote by $F _{t , \ a} \ \ (0 \leq a \leq m_t -1) $ the irreducible components. 
If $t \in {\mathbb P}^1_{\C} \setminus \Sigma ({\cal E}) $, then ${\cal E}_t = F_{t, \ 0}$ is a smooth fiber, 
and we have $\left\{ t \in {\mathbb P}_{\C} ^1 \mid m_t \geq 2 \right\} = \left\{ t \in \Sigma ({\cal E}) \mid m_t \geq 2 \right\}$. 
We fix a general fiber $C_0 := {\cal E}_{t_0} , \ t_0 \in {\mathbb P}^1_{\C} \setminus \Sigma ({\cal E})$, 
and we take $F_{t, \ 0}$ to be the unique component of ${\cal E}_t$ intersecting with $\infty$. 

We denote by $E(\C (t))_{\rm free}$ the quotient group 
$E(\C (t)) / E(\C (t))_{\rm tor}$, where 
$E(\C (t))_{\rm tor}$ is the torsion subgroup. 
Let $r$ be the Mordell-Weil rank of $E / \C (t)$, i.e., the rank of $E(\C (t))$, and we take $r$ generators $P_1, \ \cdots , \ P_r$ of $E(\C (t))_{\rm free}$. 
We put 
$$D_{i} = (P_i) - \infty  \in {\rm Div} ({\cal E}) \ \ (1 \leq i \leq r), $$
where ${\rm Div} (\cal E)$ is the group of divisors on ${\cal E}$. 
\begin{Prop} \label{Shioda}
The free part of the N\'eron-Severi group ${\rm NS } ({\cal E})$ of the elliptic surface ${\cal E}$, denoted by ${\rm NS } ({\cal E})_{\rm free}$, 
is generated by the following divisors: 
\begin{eqnarray} \label{generators}
C_0 , \ \infty , \ D_{1} , \ \dots , \ D_{r}, \ F_{t , \ a} \ \ (t \in \Sigma ({\cal E}) , \ 1 \leq a \leq m_t - 1). 
\end{eqnarray}
In particular, the Picard number $\rho $ of the elliptic surface ${\cal E}$ is given by 
$$
\rho = r + 2 + \sum_{t \in \Sigma ({\cal E})}^{} (m_t - 1). 
$$
\end{Prop}
\begin{proof}
See Theorem $1.1$ in \cite{2}. 
\end{proof}
In \cite{1}, he computed the Mordell-Weil rank $r$, but did not give $D_i$'s explicitly. 
We will give 
$r$ linearly independent points of 
the Mordell-Weil group in section 3. 
Note that 
these points are not always generators of the group. 
We end this section by introducing a practical way to show the linear 
independence of the divisors 
$C_0 , \ \infty , \ D_{1} , \ \dots , \ D_{r}, \ F_{t , \ a}$ $(t \in \Sigma ({\cal E}) , \ 1 \leq a \leq m_t - 1)$, or equivalently the intersection matrix $M$ of 
these divisors has a nonzero determinant. 

For each $P_i \in E (\C (t))$, we have 
$D_{i} \cdot C_0 = ((P_i) - \infty) \cdot C_0 = 1 - 1 = 0. $
Then there exists a fibral divisor $\Phi _{i } \in {\rm Div} ({\cal E}) \otimes \Q$ such that 
$$(D_{i} + \Phi _{i}) \cdot F = 0 \ \ {\rm for \ all \ fibral \ divisors \ } F \in {\rm Div} (\cal E). $$
More explicitly the divisor $\Phi_i$ is obtained in the following way ({\rm see} Prop. $8.3$ in \cite{4}): \\
We set $a_{t 0} (P_i)= 0$ for all $t \in {\mathbb P}_{\C}^1$. Further, when $m_t \geq 2$ we consider the following system of linear equations: 
$$\sum_{k = 1}^{m_t - 1} a_{t k } (P_i) F_{t , \ k} \cdot F_{t , \ l} = - D_{i}\cdot F_{t , \ l} \ \ (1 \leq l \leq m_t - 1) . $$
This is a system of $m_t - 1$ equations in the $m_t - 1$ variables $a_{t k} (P_i)$. 
Since 
the intersection matrix 
$\left( F_{t, \ i} \cdot F_{t, \ j} \right) _{1 \leq i, \ j \leq m_t - 1}$
has a nonzero determinant, this system of equations has a unique solution in rational numbers $a_{t k} (P_i) \in \Q$. 
Then the divisor 
\begin{eqnarray*} 
\Phi _{i} := \sum_{t \in \mathbb{P}_{\C}^1}^{} \sum_{k = 0}^{m_t - 1} a_{t k} (P_i) F_{t, \ k} 
= \sum_{t \in \{ t_1, \ \cdots, \ t_s \}}^{} \sum_{k = 1}^{m_t - 1} a_{t k}(P_i) F_{t, \ k}
\in {\rm Div} ({\cal E}) \otimes \Q
\end{eqnarray*}
has the desired property, 
where we set $\left\{ t_1, \ \cdots , \ t_s \right\} 
= \left\{ t \in \Sigma ({\cal E}) \mid m_t \geq 2 \right\}$. 
Note that since $F_{t_\alpha , \ k} \cdot F_{t_\beta , \ l} = 0 \ (\alpha \not = \beta )$, we have 
$$0 = (D_i + \Phi _i) \cdot F_{t_\alpha , \ j} = \left( D_i + \sum_{k = 1}^{m_{t_\alpha } - 1} a_{t_\alpha  k} (P_i) F_{t_\alpha , \ k}  \right) \cdot F_{t_\alpha , \ j} .$$ 
We fix a uniformizer $u_t \in \C (t)$ at $t$, that is, ${\rm ord}_t (u_t) = 1$. 
Let $f^{*} : {\rm Div} ({\mathbb P}_{\C}^1) \longrightarrow {\rm Div} ({\cal E})$ be a homomorphism 
defined by extending 
$(t) \longmapsto \sum_{j = 0}^{m_t - 1} {\rm ord}_{F_{t, \ j}} (u_t \circ f) F_{t, \ j}$ linearly. 
For each points $t_1, \ t_2 \in {\mathbb P}^1_{\C} \setminus \Sigma ({\cal E})$ with $t_1 \not = t_2$, 
since $C_0$ is algebraic equivalent to $f^{*} (t_i) \ (i = 1, \ 2)$, we have 
$$C_0^2 = f^{*} (t_1) \cdot f^{*} (t_2) = 0.$$
Now it is not hard to show the following lemma. 
\begin{Lem} \label{int.mat.}
Let $M$ be the intersection matrix of divisors 
$C_0 , \ \infty , \ D_{1} , \ \dots $, $D_{r}, \ F_{t , \ a} \ (t \in \Sigma ({\cal E}) , \ 1 \leq a \leq m_t - 1)$, and let 
$M_{\alpha } \ (1 \leq \alpha \leq s)$ be the intersection matrix of divisors $F_{t_\alpha , \ 1}, \ \cdots , \ F_{t_\alpha , \ m_{t_\alpha } - 1}$. 
Put 
\begin{eqnarray*}
N = \left[
\begin{array}{c c c}
(D_{1} + \Phi _{1}) \cdot D_{1} & \cdots &  (D_{1} + \Phi _{1} ) \cdot D_{r} \\
\vdots & \ddots & \vdots \\
(D_{r} +\Phi _{r}) \cdot D_{1} & \cdots & (D_{r} + \Phi _{r}) \cdot D_{r} \\
 \end{array}
\right] .
\end{eqnarray*}
Then we have 
$${\rm det} (M) = - {\rm det} (N) \prod_{\alpha =1}^{s} {\rm det} (M_\alpha ). $$ 
In particular, ${\rm det}(M) \not = 0$ if and only if $ {\rm det} (N) \not = 0$ since  $M_\alpha$ has a nonzero determinant for each $\alpha$. 
Note that each $M_{\alpha }$ gives one of the root lattices $A_n, \ D_n, \ E_6, \ E_7$ or $E_8$, and 
${\rm det} (M_{\alpha })$ equals the number of simple components of the singular fiber ${\cal E}_{t_{\alpha }}$. 
\end{Lem}

\section{Stiller's list of elliptic surfaces}

In this section, we give 
$\mathbb Q$-bases 
of the N\'eron-Severi groups 
of Stiller's list of elliptic surfaces 
( Example $1, \ 2, \ 3, \ 4$ and $5$ in \cite{1}) explicitly. 
If these surfaces are rational, 
then we also show that they are $\Z$-bases. 
Note that these N\'eron-Severi groups are torsion free (see Theorem $1.2$ in \cite{2.5}). 
We give a proof in detail in the case of Example $4$. 
We only give results in the other cases because the argument is the same as in Example $4$. 
In this paper, let $\zeta _n$ be a primitive $n$-th root of unity for a natural number $n$. 

\subsection{Example 4 in \cite{1}}

Recall the Example $4$ in \cite{1} which is the minimal elliptic surface whose generic fiber is the elliptic curve 
defined by an equation 
\begin{equation} \label{Ex4}
Y^2 = 4 X^3 - 3 u ^{4 n} X + u^{5 n} (u^{n} - 2) \ \ (u \in \mathbb{P}^1 _{\C}, \ n \in {\mathbb N}) 
\end{equation}
over $\C (u)$. 
We change the variables in the following way 
$$
X = \frac{2^3 (3 x + 1)}{3^6 t^{2 n}}, \ Y = \frac{-2^{5} \sqrt{2} y}{3^7 \sqrt{3} t^{3 n}} , \ u = \sqrt[n]{ \frac{- 4}{27}} t^{-1}. 
$$
Then the defining equation (\ref{Ex4}) becomes 
\begin{equation} \label{Ex4 2}
y^2 = x^3 + x^2 + t^n \ \ (t \in \mathbb{P}^1 _{\C}, \ n \in {\mathbb N}) 
\end{equation}
Let $E_n$ be the elliptic curve defined by (\ref{Ex4 2}) 
and $f : {\cal E}_n \longrightarrow {\mathbb P}_{\C} ^1$ 
be the associated elliptic surface. 
For the later use, we write down the construction of ${\cal E}_n$. 
Put 
\begin{eqnarray*}
&&\overline{X_1} = \left\{ ([X: \ Y: \ Z], \ t) \in \mathbb{P}_{\C}^2 \times \mathbb{A}^1 _t \mid Y^2 Z = X^3 + X^2 Z + t^n Z^3  \right\}, \\
&&g _1 : \overline{X_1} \longrightarrow \mathbb{A}^1 _t \  ; \  ([X: \ Y: \ Z], t) \longmapsto t .
\end{eqnarray*}
Let us write $n = 6 l + k$ with $1 \leq k \leq 6$. 
By putting $\overline{x} = x/ t^{2(l+1)} , \ \overline{y} = y / t^{3(l+1)} , \ \overline{t} = 1 / t $, 
we obtain the minimal Weierstrass form
$$\overline{y}^2 = \overline{x} ^3 + \overline{t}^{2(l+1)} \overline{x}^2 + \overline{t}^{6 - k} .$$
over $t=\infty$. 
Put 
\begin{eqnarray*}
&&\overline{X_2} = \left\{ ([\overline{X}: \ \overline{Y}: \ \overline{Z}], \ \overline{t}) \in \mathbb{P}^2 \times \mathbb{A}^1 _{\overline{t}} \mid 
\overline{Y}^2 \overline{Z} = \overline{X} ^3 + \overline{t}^{2(l+1)} \overline{X}^2 \overline{Z} + \overline{t}^{6 - k} \overline{Z}^3  \right\}, \\
&&g _2 : \overline{X_2} \longrightarrow \mathbb{A}^1 _{\overline{t}} \  ; \  ([\overline{X}: \ \overline{Y}: \ \overline{Z}], \ \overline{t}) \longmapsto \overline{t}. 
\end{eqnarray*}
By gluing the surfaces $\overline{X _1}$ and 
$\overline{X _2}$, we obtain a projective surface $W$ together with a surjective morphism 
$g : W \longrightarrow {\mathbb P}_{\C}^1$. 
The surface $W$ has singularities in $g^{-1}(0)$ and $g^{-1}(\infty)$. 
Taking the minimal resolution of singularities of $W$, we obtain ${\cal E}_n$ with $f : {\cal E}_n \longrightarrow {\mathbb P}_{\C} ^1$. 

By using Tate's algorithm (see \cite{4}), 
one can show that the surface ${\cal E}_n$ has singular fibers of type ${\rm I}_n $ over $0$, 
type ${\rm I}_1 $ over $\zeta_n ^i  \sqrt[n]{ - 4 / 27} \ \ (0 \leq i \leq n - 1)$ and 
type ${\rm II}^* $ (resp. ${\rm IV}^*, \ {\rm I}_0^*, \ {\rm IV}, \ {\rm II}, \ {\rm I}_0$) over $\infty$ as $n \equiv 1$ (resp. $2, \ 3, \ 4, \ 5, \ 0$) modulo $6$. 
Stiller computed the Mordell-Weil rank $r = {\rm rank} (E_n (\C (t)))$ and hence 
the Picard number $\rho = {\rm rank} ({\rm NS} ({\cal E}_n))$, which is given by as follows. 
\begin{center}
Table $1.$ The Mordell-Weil rank $r$ and the Picard number $\rho $
{\normalsize
\begin{tabular}{|c|l|l|}
\hline
$n$ & \multicolumn{1}{|c|}{ $r$} & \multicolumn{1}{|c|}{$\rho $} \\ \hline \hline
$6 l + 1, \ l \geq 0$ & $r = \left\{
\begin{array}{l l}
 4 & {\rm if} \ l \equiv 4 \ {\rm mod} \ 5, \\
 0 & {\rm otherwise.} \\
\end{array}
\right.$ & $\rho  = \left\{
\begin{array}{l}
 n + 13 \\
 n + 9 \\
\end{array}
\right.$ \\ \hline
$6 l + 2, \ l \geq 0$ & $r = \left\{
\begin{array}{l l}
 7 & {\rm if} \ 3l + 1 \equiv 0 \ {\rm mod} \ 10, \\
 5 & {\rm if} \ 3l + 1 \equiv 5 \ {\rm mod} \ 10, \\
 3 & {\rm if} \ 3l + 1 \equiv 2,\ 4,\ 6, \ 8 \ {\rm mod} \ 10, \\
 1 & {\rm otherwise.} \\
\end{array}
\right.$ & $\rho  = \left\{
\begin{array}{l}
 n + 14 \\
 n + 12 \\
 n + 10 \\
 n + 8 \\
\end{array}
\right.$ \\ \hline
$6 l + 3, \ l \geq 0$ & $r = \left\{
\begin{array}{l l}
 6 & {\rm if} \ l \equiv 2 \ {\rm mod} \ 5, \\
 2 & {\rm otherwise.} \\
\end{array}
\right.$ & $\rho  = \left\{
\begin{array}{l}
 n + 11 \\
 n + 7 \\
\end{array}
\right.$ \\ \hline
$6 l + 4, \ l \geq 0$ & $r = \left\{
\begin{array}{l l}
 7 & {\rm if} \ 3l + 2 \equiv 0 \ {\rm mod} \ 10, \\
 5 & {\rm if} \ 3l + 2 \equiv 5 \ {\rm mod} \ 10, \\
 3 & {\rm if} \ 3l + 2 \equiv 2,\ 4,\ 6, \ 8 \ {\rm mod} \ 10, \\
 1 & {\rm otherwise.} \\
 \end{array}
\right.$ & $\rho  = \left\{
\begin{array}{l}
 n + 10 \\
 n + 8 \\
 n + 6 \\
 n + 4 \\
\end{array}
\right.$ \\ \hline
$6 l + 5, \ l \geq 0$ & $r = \left\{
\begin{array}{l l}
 4 & {\rm if} \ l \equiv 0 \ {\rm mod} \ 5, \\
 0 & {\rm otherwise.} \\
 \end{array}
\right. $ & $\rho  = \left\{
\begin{array}{l}
 n + 5 \\
 n + 1 \\
\end{array}
\right.$ \\ \hline
$6 l + 6, \ l \geq 0$ & $r = \left\{
\begin{array}{l l}
 9 & {\rm if} \ l + 1 \equiv 0 \ {\rm mod} \ 10, \\
 7 & {\rm if} \ l + 1 \equiv 5 \ {\rm mod} \ 10, \\
 5 & {\rm if} \ l + 1 \equiv 2,\ 4,\ 6, \ 8 \ {\rm mod} \ 10, \\
 3 & {\rm otherwise.} 
 \end{array}
\right.$ & $\rho  = \left\{
\begin{array}{l}
 n + 10 \\
 n + 8 \\
 n + 6 \\
 n + 4 \\
\end{array}
\right.$ \\ \hline
\end{tabular}
}
\end{center}
The result for $r$ can be summarized in the following way. 
Put ${\rm Adm}_4 = \{ 2, \ 3, \ 4$, $5\}$. Then 
\begin{equation} \label{r4}
r = 
\sum_{
d \mid n, \ 
d \in {\rm Adm}_4
}
\varphi (d) ,
\end{equation}
where $\varphi $ is the Euler function. 
We now define $\varphi(d)$ rational points of $E_d$ 
for each $d\in {\rm Adm}_4$. 
\begin{Def}
For $d \in {\rm Adm}_4$ and $j \in (\Z/d \Z)^\times$, we define $\C (t)$-rational points $P_{d, \ j}$ of $E_d$ as follows. 
\begin{eqnarray*}
&& P_{2, \ 1} = \left( 0 , \ - t \right) , \\
&& P_{3, \ j} = \left( - \zeta _3 ^j t  , \ - \zeta _3 ^j t  \right) , \\
&& P_{4, \ 1} = \left( \sqrt{2} t + 2 t^2, \ \sqrt{2} t + 3 t^2 + 2 \sqrt{2} t^3 \right) , \\
&& P_{4, \ 3} = \left( -2 + 2 (-1)^{\frac{1}{4}} t , \ - 2 \sqrt{-1} + 4 \sqrt{-1} (-1)^{\frac{1}{4}} t + t^2 \right) , \\
&& P_{5, \ j} = \left( 2 ^{- \frac{2}{5}} \zeta _5 ^{2 j} t^{2} , \ 2 ^{-\frac{2}{5}} \zeta _5 ^{2 j} t^{2} +  2 ^{-\frac{3}{5}} \zeta _5 ^{3 j} t^{3}  \right) .
\end{eqnarray*}
\end{Def}
For $d$ which divides into $n$, 
there is the surjective map $\rho:E_n \rightarrow  E_d$ 
given by $(x, \ y, \ t) \mapsto (x, \ y, \ t^{n/d})$, and then the inverse image $\rho^*(P_{d,j})$ 
defines a $\C (t) $-rational point of $E_n$. In what follows, we use the same symbol 
$P_{d, \ j}$ for $\rho^*(P_{d, \ j})$ since it will not be confusing from the contexts. 
\begin{Th} \label{生成元}
Let $f : {\cal E}_n \rightarrow {\mathbb P}_{\C}^1 \ (n \in {\mathbb N})$ be the elliptic surfaces 
associated to the elliptic curves $E_n : y^2 = x^3 + x^2 + t^n \ (n \in {\mathbb N})$ over ${\mathbb P}_{\C}^1$. 
Then, for each $n \in {\mathbb N}$, 
${\rm NS} ({\cal E}_n)$ 
has a $\mathbb Q$-basis 
$C_0, \ \infty , \ D_{d, \ j}, \ F_{t, \ a} \ (d \in {\rm Adm}_4, \ d \mid n , \ j \in (\Z/ d \Z)^{\times}, t \in \Sigma ({\cal E}_n), 1\leq a\leq m_t - 1) $, 
where 
$D_{d, \ j} = (P_{d, \ j}) - \infty$ and 
the other notations are same as section 2. 
Moreover if ${\cal E}_n$ is rational {\rm (}i.e., $n \leq 6${\rm )}, then these divisors form a $\Z$-basis. 
\end{Th}

\begin{proof}
In the cases of $n = 6l+1$ with $l \not \equiv 4\ {\rm mod} \ 5$ or $n = 6l+5$ with $l \not \equiv 0 \ {\rm mod} \ 5$, we have $r = 0$ by Table $1$, 
so the assertion follows immediately from Proposition \ref{Shioda}. 
We consider the other cases. 
All we have to do is to show the linear independence of the divisors in Theorem \ref{生成元}. 
Let $ \ d(n) \ $ be the least common multiple of all numbers in $\{ d \in {\rm Adm}_4  \ ; \ d \mid n  \}$. 
Note that the set is non-empty in these cases. 
Since $n$ divides by $d(n)$, there exists a rational map 
$${\cal E}_n \longrightarrow {\cal E}_{d(n)} , \ \ (x, \ y, \ t) \longmapsto (x, \ y, \ t^{n/d (n)}). $$
This map induce an injection 
$$E_{d(n)} (\C (t))  \hookrightarrow  E_n (\C (t)) , \ (x(t) , \ y(t)) \longmapsto \left( x(t^{n/d (n)}) , \ y(t^{n/d (n)}) \right) .$$
We obtain ${\rm rank} (E_{d (n)} (\C (t))) = {\rm rank} (E_n (\C (t)))$ by (\ref{r4}), 
hence linearly independent points of 
$E_{d(n)} (\C (t))$ 
are those of $E_n (\C (t))$. 
There exists an injection 
$$E_{d(n)} (\C (t)) \hookrightarrow E_{60} (\C (t)) , \ (x(t) , \ y(t)) \longmapsto \left( x(t^{60/d (n)}) , \ y(t^{60/d (n)}) \right) .$$
In particular, points in $E_{d(n)} (\C (t))$ are linearly independent 
if and only if so are they in $E_{60} (\C (t))$. 
Therefore it is sufficient to show the assertion in the case of $n = 60$.
Recall that the surface ${\cal E}_{60}$ has singular fibers of type ${\rm I}_{60}$ over $0$, type ${\rm I}_{1}$ over $\zeta _{60} ^i \sqrt[60]{-4 / 27} \ \ (0 \leq i \leq 59)$ and type ${\rm I_0}$ 
over $\infty$. 

We want to show that the divisors 
$C_0 , \ \infty , \ D_{2, \ 1} , \ D_{3, \ 1} , \ D_{3, \ 2} , \ D_{4, \ 1} , \ D_{4, \ 3} $, $D_{5, \ 1} , \ D_{5, \ 2}, \ D_{5, \ 3} , \ D_{5, \ 4} , \ F_{0, \ 1}, \ \cdots , $ $F_{0, \ 59}$
are $\mathbb Q$-basis of ${\rm NS} ({\cal E} _{60} )$. 
Equivalently the matrix 
$$N = \left[
\begin{array}{c c c}
(D_{2, \ 1} + \Phi _{2, \ 1}) \cdot D_{2, \ 1} & \cdots &  (D_{2, \ 1} + \Phi _{2, \ 1} ) \cdot D_{5, \ 4} \\
\vdots & \ddots & \vdots \\
(D_{5, \ 4} +\Phi _{5, \ 4}) \cdot D_{2, \ 1} & \cdots & (D_{5, \ 4} + \Phi _{5, \ 4}) \cdot D_{5, \ 4} \\
 \end{array}
\right]$$
has a nonzero determinant (see Section $2$ for the notations). 

Firstly we compute the self intersection numbers 
$\infty ^2 , \ (P_{2, \ 1})^2 , \ (P_{3, \ i})^2$, $(P_{4, \ j})^2, \ (P_{5, \ k})^2$. 
\begin{Lem} \label{標準因子}
Let $n = 6 l + k$ with $l \geq 0 , \ 1 \leq k \leq 6$. 
Then the canonical divisor $K_{{\cal E}_{n}}$ of the elliptic surface $f :  {\cal E}_{n} \longrightarrow \mathbb{P}_{\C} ^1$ is 
$K_{{\cal E}_{n}} \cong f^* {\cal O}_{\mathbb{P}_{\C} ^1} (l -1)$, 
and we have $(P)^2 = \infty ^2 = - (l + 1)$ for each point $P \in E_n (\C (t))$. 
In particular, if $n$ equals $60$, then we have 
\begin{eqnarray} \label{self}
(P_{2, \ 1})^2 = (P_{3, \ i})^2 = (P_{4, \ j})^2 = (P_{5, \ k})^2 =  \infty ^2 = - 10.
\end{eqnarray}
\end{Lem}

\begin{proof} 
By Kodaira's canonical bundle formula (see \cite{5}), we get
\begin{eqnarray} \label{d=}
K_{{\cal E}_{n}} \cong f^* {\cal O}_{\mathbb{P}_{\C} ^1} (d -2), \ \ \ d = \frac{1}{12} \sum_{t \in \Sigma ({\cal E}_{n})}^{} \varepsilon (t), 
\end{eqnarray}
where $\varepsilon (t)$ is defined as follows: 
\begin{center}
\begin{tabular}{|c|c|c|c|c|c|c|c|c|}
\hline
type & ${\rm I}_n$ & II & III & IV & ${\rm I }_n ^*$ & ${\rm II} ^*$ & ${\rm III} ^*$ & ${\rm IV} ^*$  \\ \hline 
$\varepsilon (t) $ & $n$ & $2$ & $3$ & $4$ & $n + 6$ & $10$ & $9$ & $8$ \\
\hline
\end{tabular}
\end{center}
By the reduction type of the singular fibers of ${{\cal E}_{n}}$,  
we have
$d 
= l + 1 . $
Thus 
\begin{eqnarray} \label{canonical}
K_{{\cal E}_{n}} \cong f^* {\cal O}_{\mathbb{P}_{\C} ^1} (d -2) = f^* {\cal O}_{\mathbb{P}_{\C} ^1} (l -1).
\end{eqnarray}
By (\ref{isom}), each point $P \in E_n (\C (t))$ corresponds to a section $\sigma _P : {\mathbb P}_{\C}^1 \longrightarrow {\cal E}_n$. 
The translation-by-$P$ map on $E_n$ can be uniquely extended to a map $\tau_P : {\cal E}_n \longrightarrow  {\cal E}_n$ by the minimality of ${\cal E}_n$ (see Prop. $9.1$ in \cite{4}). 
It follows that $\tau _P^* D_1 \cdot \tau _P^* D_2 = D_1 \cdot D_2$ for any two divisors $D_1, \ D_2 \in {\rm Div} ({\cal E}_n)$. Hence 
$(P)^2 = \tau_P ^* (P) \cdot \tau_P ^* (P) = \infty ^2 $. 

Since $\infty$ is isomorphic to $\mathbb{P}^1 _{\C}$, $\infty$ is of genus zero.  Thus, by the adjunction formula, we get 
$$\frac{1}{2} \left( \infty ^2 + K_{{\cal E}_{n}} \cdot \infty \right) + 1 = 0, \ \ {\rm i.e., \ \ } \infty ^2 = - \left( K_{{\cal E}_{n}} \cdot \infty + 2 \right). $$
On the other hand, 
we can compute
\begin{eqnarray*}
K_{{\cal E}_{n}} \cdot \infty &=& (f^* {\cal O} _{\mathbb{P}^1} (l - 1)) \cdot \infty = (l - 1) f^* ({\cal O} _{\mathbb{P}^1} (1)) \cdot \infty \\  
&=& (l - 1) C_0 \cdot \infty = l - 1 
\end{eqnarray*}
by (\ref{canonical}). Therefore we get $(P)^2 = \infty ^2 = - (l + 1)$ for all $P \in E_n (\C (t))$.  
\end{proof}

Next we compute the intersection numbers of divisors 
$\infty , \ (P_{2, \ 1}), \ (P_{3, \ i})$, 
$(P_{4, \ j}), \ (P_{5, \ k}) , \ F_{0, \ a} \ \ (1 \leq a \leq 59)$ 
in ${\cal E}_{60}$. 
In the affine surface $X_1 : y^2 = x^3 + x^2 + t^{60}$, divisors $(P_{2, \ 1}) , \ (P_{3, \ i}) , \ (P_{4, \ j}) \ {\rm and} \ (P_{5, \ k})$ are given by 
\begin{eqnarray*}
(P_{2, \ 1}) &=& \left( x = 0 , \ y = - t^{30} \right) , \\ 
(P_{3, \ i}) &=& \left( x = - \zeta _3 ^i t ^{20} , \ y = - \zeta _3 ^i t ^{20} \right) , \\
(P_{4, \ 1}) &=& \left( x = \sqrt{2} t^{15} + 2 t^{30}, \ y = \sqrt{2} t^{15} + 3 t^{30} + 2 \sqrt{2} t^{45} \right) , \\
(P_{4, \ 3}) &=& \left( x = -2 + 2 (-1)^{\frac{1}{4}} t^{15} , \ y = - 2 \sqrt{-1} + 4 \sqrt{-1} (-1)^{\frac{1}{4}} t^{15} + t^{30} \right) , \\
(P_{5, \ k}) &=& \left( x = 2 ^{- \frac{2}{5}} \zeta _5 ^{2 k} t^{24} , \ y = 2 ^{- \frac{2}{5}} \zeta _5 ^{2 k} t^{24} + 2 ^{- \frac{3}{5}} \zeta _5 ^{3 k} t^{36} \right) .
\end{eqnarray*}
Thus $(P_{4, \ 3})$ does not pass through the singular point $(0, \ 0, \ 0)$ of 
$X_1$, however 
$(P_{2, \ 1}), \ (P_{3, \ i}) , \ (P_{4, \ 1})$ and $(P_{5, \ k})$ pass through this point. 
Since this singular point is of type $A_{59}$, we can resolve the singular point by the blowing-up $30$ times. 
We denote by $(x_{(m)} , \ y_{(m)} , \ t_{(m)})$ the coordinates in the neighborhood of the singular point after the $m$-th blowing-up $(1 \leq m \leq 29)$, 
and denote by $(P_{2, \ 1})^{(m)}, \ (P_{3, \ i})^{(m)} , \ (P_{4, \ 1})^{(m)} , \ (P_{5, \ k})^{(m)}$ 
the $m$-th blowing-up of $(P_{2, \ 1}), \ (P_{3, \ i}) , \ (P_{4, \ 1}) , \ (P_{5, \ k})$, respectively. 
These curves are given by 
\begin{eqnarray*}
&&(P_{2, \ 1})^{(m)} = \left( x_{(m)} = 0 , \ y_{(m)} = - t_{(m)} ^{30 - m} \right) , 
\\
&&(P_{3, \ i}) ^{(m)} = \left( x_{(m)} = - \zeta _3 ^i t_{(m)} ^{20 - m} , \ y_{(m)} = \zeta _3 ^i t_{(m)} ^{20 - m} \right) ,
\end{eqnarray*}
\begin{eqnarray*}
&&(P_{4, \ 1}) ^{(m)} = \Big( x_{(m)} = \sqrt{2} t^{15 - m} + 2 t^{30 - m} , \\
&&\hspace{5.2cm} y_{(m)} = \sqrt{2} t^{15 - m} + 3 t^{30 - m} + 2 \sqrt{2} t^{45- m} \Big) ,
\\
&&(P_{5, \ k}) ^{(m)} = \left( x_{(m)} = 2 ^{- \frac{2}{5}} \zeta _5 ^{2 k} t_{(m)}^{24 - m} , \ y_{(m)} = 2 ^{- \frac{2}{5}} \zeta _5 ^{2 k} t_{(m)}^{24 - m} + 2 ^{- \frac{3}{5}} \zeta _5 ^{3 k} t_{(m)}^{36 - m} \right) .
\end{eqnarray*}
In particular, in ${\cal E}_{60}$ divisors $(P_{3, \ i})$ (resp. $(P_{4, \ 1}), \ (P_{5, \ k})$) intersect with either of two ${\mathbb P}_{\C}^1$ which appear 
by the $20$-th (resp. $15, \ 24$-th) blowing-up, 
and divisors $(P_{2, \ 1})$ intersect with unique ${\mathbb P}_{\C}^1$ which appears by the $30$-th blowing-up. 
Hence we may assume that 
$(P_{2, \ 1})$ (resp. $(P_{3, \ i}), \ (P_{4, \ 1}), \ (P_{4, \ 3}), \ (P_{5, \ k})$) intersect with $F_{0, \ 30}$ 
(resp. $F_{0, \ 20}, \ F_{0, \ 15}, \ F_{0, \ 0}, \ F_{0, \ 24}$).  
In addition, in $X_1 \cap (t \not = 0)$ 
\begin{eqnarray*}
&&(P_{2, \ 1}) \cap (P_{3, \ i})
= \emptyset , \\
&&(P_{2, \ 1}) \cap (P_{4, \ 1}) 
= \left\{ \left( 0 , \ - \frac{1}{2} , \ t \right) \mid t^{15} = - \frac{1}{\sqrt{2}} \right\}, \\
&&(P_{2, \ 1}) \cap (P_{4, \ 3}) 
= \left\{ \left( 0 , \ - \sqrt{-1} , \ t \right) \mid t^{15} = (-1)^{\frac{3}{4}} \right\}, \\
&&(P_{2, \ 1}) \cap (P_{5, \ k}) 
= \emptyset , \\ 
&&(P_{3, \ 1}) \cap (P_{3, \ 2}) = \emptyset , \\
&&(P_{3, \ i}) \cap (P_{4, \ 1}) 
= \left\{ \left( \frac{1}{\sqrt{2}}, \ \frac{1}{\sqrt{2}} , \ t \right) \mid t^{5} = - \zeta _3 ^{2 i} \frac{1}{\sqrt{2}}  \right\}, \\
&&(P_{3, \ i}) \cap (P_{4, \ 3}) 
= \left\{ (x (t) , \ y(t) , \ t) \mid \zeta _3^{2 i} t^{10} - (1 + \sqrt{-1}) ((-1)^{\frac{1}{4}} \zeta _3^i t^5 - 1) = 0 \right\}, \\
&&(P_{3, \ i}) \cap (P_{5, \ k}) 
= \emptyset , \\ 
&&(P_{4, \ 1}) \cap (P_{4, \ 3}) = \emptyset , \\
&&(P_{4, \ 1}) \cap (P_{5, \ k}) 
= \left\{ (x(t) , \ y(t) , t) \mid 2^{\frac{7}{10}} \zeta _5^{3 k} t^6 + \zeta _5 ^{4 k} t^3 + 2^{\frac{3}{10}} = 0 \right\}, \\
&&(P_{4, \ 3}) \cap (P_{5, \ k}) 
= \Big\{ (x(t) , \ y(t) , t) \mid 
\zeta _5^{3 k} t^6 \left( \zeta _5^{3 k} t^6 - 2^{\frac{4}{5}} (-1)^{\frac{1}{4}} \zeta _5^{4 k} t^3 + 2^{\frac{3}{5}} \sqrt{-1} \right)  
\\
&& \hspace{4.7cm} 
- 2^{\frac{1}{5}} (1 + \sqrt{-1}) \left( 2^{\frac{1}{5}} (-1)^{\frac{1}{4}} \zeta _5 ^{4 k} t^3 - \sqrt{-1} \right) = 0 \Big\}, \\
&&(P_{5, \ k_1}) \cap (P_{5, \ k_2}) = \emptyset \ \ (k_1 \not = k_2) . 
\end{eqnarray*}
and local intersection numbers of divisors $(P_{2, \ 1}), \ (P_{3, \ i}), \ (P_{4, \ j}), \ (P_{5, \ k})$ at these intersection points are all one.  

On the other hand, in the $\infty$-model $X_2$ (i.e., the surface obtained by the variable transformation
$\overline{x} = x / t^{2 0}, \ \overline{y} = y / t^{3 0}, \ \overline{t} = 1 / t$ ) or its projection $\overline{X_2}$, 
divisors $(P_{2, \ 1}), \ (P_{3, \ i}), \ (P_{4, \ j})$ and $(P_{5, \ k})$ are given by 
\begin{eqnarray*}
(P_{2, \ 1}) &=& \left( \bar{x} = 0 , \ \bar{y} = - 1  \right) , \\
(P_{3, \ i}) &=& \left( \bar{x} = - \zeta _3 ^i , \ \bar{y} = - \zeta _3 ^i \overline{t} ^{10} \right) , \\
(P_{4, \ 1}) &=&
 \left\{ \left( \left[ \sqrt{2}  \overline{t}^{20} + 2 \overline{t}^{5} : \ \sqrt{2}  \overline{t}^{30} + 3 \overline{t}^{15} +2^{\frac{2}{3}} :  \overline{t} ^{15} \right] ,  \ \overline{t} \right) 
\mid \overline{t} \in \mathbb{A}^1 \right\} , \\ 
(P_{4, \ 3}) &=& \left(  \bar{x} = -2  \bar{t}^{20} + 2 (-1)^{\frac{1}{4}} \bar{t}^{5} , \  \bar{y} = - 2 \sqrt{-1} \bar{t}^{30} + 4 \sqrt{-1} (-1)^{\frac{1}{4}} \bar{t}^{15} + 1 \right) , \\
(P_{5, \ k}) &=& \left\{ \left( \left[ 2 ^{- \frac{2}{5}} \zeta _5 ^{2 k} \overline{t} ^{2}  : \ 2 ^{- \frac{2}{5}} \zeta _5 ^{2 k} \overline{t}^{12}  + 2 ^{- \frac{3}{5}} \zeta _5 ^{3 k} 
:  \overline{t} ^{6} \right] , \ \overline{t} \right) 
\mid \overline{t} \in \mathbb{A}^1 \right\}. 
\end{eqnarray*}
Thus when $\overline{t}$ equal $0$ (i.e., $t$ equal $\infty$), divisors $(P_{4, \ 1}), \ (P_{5, \ k})$ and $\infty$ intersect at $([0: \ 1: \ 0], \ 0) \in \overline{X_2}$ and 
the other pairs of 
$(P_{2, \ 1}), \ (P_{3, \ i}), \ (P_{4, \ j}), \ (P_{5, \ k})$ and $\infty$ do not intersect.  
Moreover the local intersection number of $(P_{4, \ 1})$ and $\infty$ at this point is five and 
the 
numbers of the other pairs of $(P_{4, \ 1}), \ (P_{5, \ k})$ and $\infty$ 
is two. 
From the above and (\ref{self}), we obtain 
\begin{eqnarray*}
(P_{2, \ 1}) \cdot (P_{d, \ l}) &=& \left\{
\begin{array}{c l}
-10 & (d = 2), \\
 0 & (d = 3, \ 5), \\
 15 & (d = 4), \\
 \end{array}
\right. 
\\
(P_{3, \ i}) \cdot  \ (P_{d, \ l}) &=& \left\{
\begin{array}{c l}
-10 & (d = 3, \ i = l), \\
 0 & (d = 3, \ i \not = l), \\
 5 & (d = 4, \ j = 1), \\
 10 & (d = 4, \ j = 3), \\
 0 & (d = 5), \\
 \end{array}
\right. 
\\
(P_{4, \ j}) \cdot (P_{d, \ l}) &=& \left\{
\begin{array}{c l}
-10 & (d = 4, \ j = l), \\
 0 & (d = 4, \ j \not = l), \\
 8 & (d = 5, \ j = 1), \\
 12 & (d = 5, \ j = 3), \\
\end{array}
\right. 
\\
(P_{5, \ k}) \cdot (P_{5, \ l}) &=& \left\{
\begin{array}{c l}
-10 & (k = l), \\
 2 & (k \not = l).  \\
\end{array}
\right.
\end{eqnarray*}

Finally we give $\Phi _{d, \ l}$ for $d \in {\rm Adm}_4, \ l \in (\Z / d \Z)^{\times}$ 
by the method mentioned in section $2$
and we compute $(D_{d, \ l} + \Phi _{d, \ l}) \cdot D_{d' , \ l'} \ (d, \ d' \in {\rm Adm}_4, \ l \in (\Z / d \Z)^{\times}, \ l' \in (\Z / d' \Z)^{\times})$. 
Recall that $\Phi _{d, \ l}$ is defined by $\Phi _{d, \ l} = \sum_{i = 1}^{59} a_i (P_{d, \ l}) F_{0, \ i}$ and 
\begin{eqnarray*}
&&(a_1 (P_{d, \ l}) , \ \cdots , \ a_{59} (P_{d, \ l})) \\
&&\hspace{2cm} = -( D_{d, \ l} \cdot F_{0, \ 1} , \ \cdots , \ D_{d, \ l} \cdot F_{0, \ 59} ) \left( F_{0, \ i} \cdot F_{0, \ j} \right) ^{-1} .
\end{eqnarray*}
For integers 
$1 \leq m \leq 59$, 
since $\infty \cdot F_{0, \ m} = 0$, 
we have 
\begin{eqnarray*}
D_{2, \ 1} \cdot F_{0, \ m} &=& (P_{2, \ 1}) \cdot F_{0, \ m} 
= \left\{
\begin{array}{c c}
 1 & {\rm if} \ m = 30, \\
 0 & {\rm if} \ m \not = 30, \\
 \end{array}
\right. \\
D_{3, \ i} \cdot F_{0, \ m} &=& (P_{3, \ i}) \cdot F_{0, \ m} 
= \left\{
\begin{array}{c c}
 1 & {\rm if} \ m = 20, \\
 0 & {\rm if} \ m \not = 20, \\
 \end{array}
\right. \\
D_{4, \ 1} \cdot F_{0, \ m} &=& (P_{4, \ 1}) \cdot F_{0, \ m} 
= \left\{
\begin{array}{c c}
 1 & {\rm if} \ m = 15,\\
 0 & {\rm if} \ m \not = 15, \\
 \end{array}
\right. \\
D_{4, \ 3} \cdot F_{0, \ m} &=& (P_{4, \ 3}) \cdot F_{0, \ m} 
= 0, \\
D_{5, \ k} \cdot F_{0, \ m} &=& (P_{5, \ k}) \cdot F_{0, \ m} 
= \left\{
\begin{array}{c c}
 1 & {\rm if} \ m = 24, \\
 0 & {\rm if} \ m \not = 24. \\
 \end{array}
\right. 
\end{eqnarray*}
Since the reduction type of $({\cal E}_{60})_0$ 
is ${\rm I}_{60}$, 
we have the intersection matrix 
$$\left( F_{0, \ i} \cdot F_{0, \ j} \right) _{1 \leq i, \ j \leq 59} = \left[
\begin{array}{c c c c}
 -2 & 1 &  \cdots & 0 \\
 1 & -2 &  \ddots & \vdots \\
 \vdots & \ddots & \ddots & 1 \\
 0 & \cdots & 1 & -2 \\
 \end{array}
\right]. $$
It is an easy exercise to show that 
the $j$-th row of the inverse matrix of this matrix is 
$\frac{-1}{60} \left[ 60 - j , \ 2 \cdot ( 60 - j) , \ \cdots, \  j \cdot ( 60 - j ), \ j \cdot (59 - j) , \ \cdots j \cdot 2 , \ j \right]. $
Therefore we obtain $a_1 (P_{4, \ 3}) = \cdots = a_{59} (P_{4, \ 3}) = 0$ and 
\begin{eqnarray*}
\left( a_1 (P_{2, \ 1}), \ \cdots , \ a_{59} (P_{2, \ 1} \right) &=& \frac{1}{60} \left[ 30 , \ 2\cdot 30 , \ \cdots , \ 30\cdot 30 , \ \cdots , \ 30\cdot 2 , \ 30 \right] , 
\end{eqnarray*}
\begin{eqnarray*}
\left( a_1 (P_{3, \ i}) , \  \cdots , \  a_{59} (P_{3, \ i}) \right) &=& \frac{1}{60} \left[ 40 , \ 2\cdot 40 , \ \cdots , \ 20\cdot 40 , \ \cdots , \ 20\cdot 2 , \ 20 \right] , \\
\left( a_1 (P_{4, \ 1}) , \  \cdots , \  a_{59} (P_{4, \ 1}) \right) &=& \frac{1}{60} \left[ 45 , \ 2\cdot 45 , \ \cdots , \ 15\cdot 45 , \ \cdots , \ 15\cdot 2 , \ 15 \right] , \\
\left( a_1 (P_{5, \ k}) , \  \cdots , \  a_{59} (P_{5, \ k}) \right) &=& \frac{1}{60} \left[ 36 , \ 2\cdot 36 , \ \cdots , \ 24\cdot 36 , \ \cdots , \ 24\cdot 2 , \ 24 \right] . 
\end{eqnarray*}
In particular, $\Phi _{4, \ 3}$ equal $0$ and $\Phi _{3, \ i}$ (resp. $\Phi _{5, \ k}$) does not depend on $i$ (resp. $k$), 
so we put $\Phi _{2} = \Phi _{2, \ 1}, \ \Phi _{3} = \Phi _{3, \ i}, \ \Phi _{4} = \Phi _{4, \ 1}, \ \Phi _{5} = \Phi _{5, \ k}$. 
We can compute $(D_{d, \ l} + \Phi _{d}) \cdot D_{d' , \ l'} \ \ (d, \ d' \in {\rm Adm}_4, \ l \ \in (\Z / d \Z)^{\times}, \ l' \in (\Z / d' \Z)^{\times})$ as follows. 
\begin{eqnarray*}
(D_{2, \ 1} + \Phi _{2}) \cdot D_{d, \ l} 
&=& \left\{
\begin{array}{c l}
-5 & (d = 2), \\
 0 & (d = 3), \\
 \frac{15}{2} & (d = 4, \ l = 1), \\
 5 & (d = 4, \ l = 3), \\
 0 & (d = 5), \\
 \end{array}
\right. 
\\
(D_{3, \ i} + \Phi _3) \cdot D_{d, \ l} 
&=& \left\{
\begin{array}{c l}
 - \frac{20}{3} & (d = 3 , \ i = l), \\
 \frac{10}{3} & (d = 3, \ i \not = l), \\
 0 & (d = 4, \ 5), \\
 \end{array}
\right. \\
(D_{4, \ j} + \Phi _{4}) \cdot D_{d, \ l} 
&=& \left\{
\begin{array}{c l}
 - \frac{75}{4} & (d = 4, \ j = l = 1), \\
 - 20 & (d = 4, \ j = l = 3), \\
 - 15 & (d = 4, \ j \not = l ), \\
 0 & (d = 5), \\
 \end{array}
\right. \\
(D_{5, \ k} + \Phi _{5}) \cdot D_{d, \ l} 
&=& \left\{
\begin{array}{c l}
 -\frac{48}{5} & (d = 5, \ k = l), \\
 \frac{12}{5} & (d = 5, \ k \not = l). \\
 \end{array}
\right. 
\end{eqnarray*}
Note that since $(D_{d, \ l} + \Phi _d) \cdot F = 0$ for all fibral divisors $F$, we have 
$(D_{d, \ l} + \Phi _{d}) \cdot D_{d' , \ l'} = 
(D_{d', \ l'} + \Phi _{d'}) \cdot D_{d , \ l}$. 
Thus we obtain 
$$\displaystyle N = 
 \left[
\begin{array}{c c c c c c c c c}
-5 & 0 & 0 & \frac{15}{2} & 5 & 0 & 0 & 0 & 0 \\
0 & -\frac{20}{3} & \frac{10}{3} & 0 & 0 & 0 & 0 & 0 & 0 \\
0 & \frac{10}{3} & -\frac{20}{3} & 0 & 0 & 0 & 0 & 0 & 0 \\
\frac{15}{2} & 0 & 0 & -\frac{75}{4} & -15 & 0 & 0 & 0 & 0 \\
5 & 0 & 0 & - 15 & -20 & 0 & 0 & 0 & 0 \\
0 & 0 & 0 & 0 & 0 & -\frac{48}{5} & \frac{12}{5} & \frac{12}{5} & \frac{12}{5} \\
0 & 0 & 0 & 0 & 0 & \frac{12}{5} & -\frac{48}{5} & \frac{12}{5} & \frac{12}{5} \\
0 & 0 & 0 & 0 & 0 & \frac{12}{5} & \frac{12}{5} & -\frac{48}{5} & \frac{12}{5} \\
0 & 0 & 0 & 0 & 0 & \frac{12}{5} & \frac{12}{5} & \frac{12}{5} & -\frac{48}{5} \\
 \end{array}
\right] $$
and 
${\rm det} (N) = -2^8 3^5 5^4 \not = 0$. 
Therefore the divisors $C_0 , \ \infty , \ D_{2, \ 1} , \ D_{3, \ i} , \ D_{4, \ j}$, $D_{5, \ k}, \ F_{0, \ 1}, \ \cdots , $ $F_{0, \ 59}$
are $\mathbb Q$-basis of ${\rm NS} ({\cal E} _{60} )$. 

Similarly in the cases of $n \leq 6$, we can compute ${\rm det} (M) = \pm 1$, where $M$ is the intersection matrix of the divisors in Theorem \ref{生成元}  
(see Lemma \ref{int.mat.} ). 
In particular, the divisors form a $\mathbb Z$-basis of ${\rm NS} ({\cal E} _{n} )$. 
\end{proof}

\subsection{Example 1 in \cite{1}}
Example $1$ in \cite{1} is the minimal elliptic surface whose generic fiber is the elliptic curve defined by 
$$
Y^2 = 4 X^3 - 3 u ^{3 n} X - u^{5 n} \ \ (u \in \mathbb{P}^1 _{\C}, \ n \in {\mathbb N})
$$
over $\C (u)$. 
By changing the variables suitably, the defining equation becomes 
\begin{equation} \label{Ex1}
y^2 = x^3 + t^n x + t^n \ \ (t \in {\mathbb P}^1_{\C}). 
\end{equation}
We denote by $E_n$ the elliptic curve defined by (\ref{Ex1}) and by $f : {\cal E}_n \longrightarrow {\mathbb P}_{\C} ^1$ 
the associated elliptic surface. 
In \cite{1}, he proved that 
the Mordell-Weil rank $r = {\rm rank} (E_n (\C (t)))$
is given by 
\begin{equation} \label{r1}
r = 
\sum_{
d \mid n, \ 
 d \in {\rm Adm}_1 
}
\varphi (d), 
\end{equation}
where 
$\varphi $ is the Euler function and ${\rm Adm}_1 = \{ 1, \ 2, \ 3, \ 7, \ 8, \ 10, \ 12, \ 15, \ 18, \ 20, \ 42\}$. 
\begin{Rem} \label{Rem}
For the use in section 4, we write down the property of ${\rm Adm}_1$. 
$d \in {\rm Adm}_1$ if and only if each $j \in \{  j \in {\mathbb N} \ ; \ 9 d \leq 12 j \leq 10 d \}$ is not relatively prime to $d$. 
Such $d$'s are called $admissible$ in \cite{1}. 
\end{Rem}

\begin{Def} \label{sections 1}
For $d \in {\rm Adm}_1$ and $j \in (\Z / d \Z) ^{\times}$, we define $\C (t)$-rational points $P_{d, \ j}$ of $E_d$ as follows. 
\begin{eqnarray*}
&& P_{1, \ 1} = (-1, \ \sqrt{-1}) , \\
&& P_{2, \ 1} = \left( \sqrt{-1} t, \ - t \right) , \\
&& P_{3, \ j} = \left( - \zeta _3 ^j t , \ \sqrt{-1} \zeta _3 ^{2 j} t^{2} \right) , \\
&& P_{7, \ j} = \left( - \zeta _7 ^{2 j} t^{2} - \zeta _7 ^{3 j} t^{3} , \  
\sqrt{-1} \left( \zeta _7 ^{3 j} t^{3} + \zeta _7 ^{4 j} t^{4} + \zeta _7 ^{5 j} t^{5} \right) \right) , \\
&& P_{8, \ j} = \big( a_0(8,j) t^{2} + a_1(8,j) t^{3} + a_2(8,j) t^{4} , 
\\ && \hspace{3cm} 
b_0(8,j) t^{3} + b_1(8,j) t^{4} + b_2(8,j) t^{5} + b_3(8,j) t^{6} \big) ,
\\
&& P_{10, \ j} = \left( 2^{\frac{2}{5}} \zeta _{10} ^{4 j}  t^{4} , \ - \zeta _{10} ^{5 j} t^{5} - 2^{\frac{1}{5}} \zeta _{10} ^{7 j} t^{7} \right) ,
 \\
&& P_{12, \ j} = 
\big( a_0(12,j) t^{4} + a_1(12,j) t^{5} + a_2(12,j) t^{6} , \\ 
&& \hspace{3cm} b_0(12,j) t^{6} + b_1(12,j) t^{7} + b_2(12,j) t^{8} + b_3(12,j) t^{9} \big) ,
\\
&& P_{15, \ j} = \Big( - \zeta _{15} ^{5 j} t^{5} - 3^{\frac{1}{5}} \zeta _{15} ^{6 j} t^{6} - 3^{\frac{2}{5}} \zeta _{15} ^{7 j} t^{7} , \\ 
&& \hspace{2cm}  \sqrt{-1} \left( 3^{\frac{3}{5}} \zeta _{15} ^{8 j} t^{8} +  3^{\frac{4}{5}} \zeta _{15} ^{9 j} t^{9} 
+ 2 \zeta _{15} ^{10 j} t^{10} +  3^{\frac{1}{5}} \zeta _{15} ^{11 j} t^{11} \right) \Big),  
\\
&& P_{18, \ j} = 
\big( a_0(18,j) \zeta _{18} ^{6 j} t^{6} + a_1(18,j) \zeta _{18} ^{8 j} t^{8} + a_2(18,j)  \zeta _{18} ^{10 j} t^{10} , \\ 
&& \hspace{2cm} b_0(18,j) \zeta _{18} ^{9 j} t^{9} + b_1(18,j) \zeta _{18} ^{11 j} t^{11} \\
&& \hspace{3cm} + b_2(18,j) \zeta _{18} ^{13 j} t^{13} + b_3(18,j) \zeta _{18} ^{15 j} t^{15} \big) ,
 \\
&& P_{20, \ j} = \Big( a_0(20,j) \zeta _{20} ^{ 6 j } t^{6} + a_1(20,j) \zeta _{20} ^{ 8 j } t^{8} + a_2(20,j) \zeta _{20} ^{ 10 j } t^{10} , \\ 
&&\hspace{2cm} b_0(20,j) \zeta _{20} ^{ 9 j } t^{9} + b_1(20,j) \zeta _{20} ^{ 11 j } t^{11} \\
&&\hspace{3cm}  + b_2(20,j) \zeta _{20} ^{ 13 j } t^{13} + b_3(20,j) \zeta _{20} ^{ 15 j } t^{15} \Big) , 
\\
&& P_{42, \ j} = \Big( a_1 \zeta _{42} ^{14 j} t^{14} + a_2 \zeta _{42} ^{16 j} t^{16} + a_3 \zeta _{42} ^{18 j} t^{18} 
+ a_4 \zeta _{42} ^{20 j} t^{20} + a_5 \zeta _{42} ^{22 j} t^{22} , \\
&& \hspace{2cm}b_1 \zeta _{42} ^{21 j} t^{21} + b_2 \zeta _{42} ^{23 j} t^{23} + b_3 \zeta _{42} ^{25 j} t^{25} \\
&&\hspace{3cm}+ b_4 \zeta _{42} ^{27 j} t^{27} + b_5 \zeta _{42} ^{29 j} t^{29} + b_6 \zeta _{42} ^{31 j} t^{31} + b_7 \zeta _{42} ^{33 j} t^{33} \Big), \\
\end{eqnarray*}
where the coefficients $a_k(d, j), \ b_k(d, j)$ are given by 
\begin{landscape}
{\small
\begin{center}
\begin{tabular}{|c|c|c|c|c|}
\hline
$j$ & $1$ & $3$ & $5$ & $7$ \\ \hline \hline 
$a_0(8, j)$ & $2^{-\frac{1}{2}}$ & $- \sqrt{2}$ & $\sqrt{-1} + (-1)^{\frac{1}{4}} + 1$ & $0$ \\ \hline 
$a_1(8, j)$ & $0$ & $2^{\frac{3}{4}}$ & $0$ & $a_2(8, 7)^{\frac{1}{4}} (2^{\frac{5}{2}} + 4 )$ \\ \hline 
$a_2(8, j)$ & $0$ & $-1$ & $\sqrt{-1}$ & $\sqrt{- 2 \sqrt{2} - 3}$ \\ \hline 
$b_0(8, j)$ & $2^{-\frac{3}{4}}$ & $- 2^{\frac{3}{4}} \sqrt{-1}$ & $(\sqrt{-1} + (-1)^{\frac{1}{4}} + 1)^{\frac{3}{2}}$ & $0$ \\ \hline 
$b_1(8, j)$ & $0$ & $3 \sqrt{-1}$ & $0$ & $1$ \\ \hline 
$b_2(8, j)$ & $2^{-\frac{1}{4}}$ & $-2^{\frac{5}{4}} \sqrt{-1}$ & $0$ & $a_2(8, 7)^{\frac{3}{4}} \sqrt{2 \sqrt{2} + 2} (\sqrt{2} + 1) ^{\frac{1}{4}}$ \\ \hline 
$b_3(8, j)$ & $0$ & $\sqrt{-2}$ & $\frac{3 (-1)^{\frac{1}{4}} (\sqrt{-1} - 1) - 4}{(\sqrt{-1} + (-1)^{\frac{1}{4}} + 1)^{\frac{3}{2}}}$ &
$a_2(8, 7)^{\frac{3}{2}} \sqrt{2 \sqrt{2} + 2} (\sqrt{2} - 1) ^{\frac{1}{4}}$ \\ \hline 
\hline
$j$ & \multicolumn{2}{|c|}{$1$} & \multicolumn{2}{|c|}{$5$} \\ \hline \hline
$a_0(12, j)$ & \multicolumn{2}{|c|}{$- \frac{1}{12} a_1(12, 1)^2 a_2(12, 1) (-33 + 5 a_2(12, 1)^2)$} & \multicolumn{2}{|c|}{$-1$} \\ \hline 
$a_1(12, j)$ & \multicolumn{2}{|c|}{a root of $z^6 + 180 a_2(12, 1) + 388 a_2(12, 1)^3 = 0$} & \multicolumn{2}{|c|}{$0$} \\ \hline 
$a_2(12, j)$ & \multicolumn{2}{|c|}{$\sqrt{2 \sqrt{3} + 3}$}& \multicolumn{2}{|c|}{$\sqrt{\frac{2}{3} \sqrt{3} - 1}$} \\ \hline \hline
$j$ & \multicolumn{2}{|c|}{$7$} & \multicolumn{2}{|c|}{$11$} \\ \hline \hline 
$a_0(12, j)$ & \multicolumn{2}{|c|}{$\frac{1 + \sqrt{-3}}{2}$} & \multicolumn{2}{|c|}{$\frac{1}{4} a_1(12, 11)^2 a_2(12, 11) (5 + 3 a_2(12, 11) ^2)$} \\ \hline 
$a_1(12, j)$ & \multicolumn{2}{|c|}{$0$} & \multicolumn{2}{|c|}{a root of $z^6 -12 a_2(12, 11) + 36 a_2(12, 11)^3 = 0$} \\ \hline 
$a_2(12, j)$ & \multicolumn{2}{|c|}{$\sqrt{\frac{2}{3} \sqrt{3} - 1}$} & \multicolumn{2}{|c|}{$\sqrt{\frac{2}{3} \sqrt{3} - 1}$} \\ \hline \hline 
$b_0(12, j)$ & \multicolumn{4}{|c|}{$ \frac{4 a_0(12, j) a_1(12, j) a_2(12, j) (a_2(12, j)^2 + 1) (3 a_2(12, j)^4 + 6 a_2(12, j)^2 - 1) - 
a_1(12, j)^3 (a_2(12, j)^2 - 1) (a_2(12, j)^4 + 6 a_2(12, j)^2 + 1)}{16 (a_2(12, j)^3 + a_2(12, j))^2 \sqrt{a_2(12, j)^3 + a_2(12, j)}}$} \\ \hline 
$b_1(12, j)$ & \multicolumn{4}{|c|}{$\frac{4 a_0(12, j) a_2(12, j) (a_2(12, j)^2 + 1) (3 a_2(12, j)^2 + 1) 
+ a_1(12, j)^2 (3 a_2(12, j)^4 + 6 a_2(12, j)^2 - 1)}{8 (a_2(12, j)^3 + a_2(12, j)) \sqrt{a_2(12, j)^3 + a_2(12, j)}}$} \\ \hline 
$b_2(12, j)$ & \multicolumn{4}{|c|}{$\frac{a_1(12, j) (3 a_2(12, j)^2 + 1)}{\sqrt{a_2(12, j)^3 + a_2(12, j)}}$} \\ \hline 
$b_3(12, j)$ & \multicolumn{4}{|c|}{$\sqrt{a_2(12, j)^3 + a_2(12, j)}$} \\ \hline 
\hline
$j$ & \multicolumn{2}{|c|}{$1, \ 11$} & \multicolumn{2}{|c|}{$5, \ 13$} \\ \hline \hline
$a_0(18, j)$ & \multicolumn{2}{|c|}{$0$} & \multicolumn{2}{|c|}{$2 b_1(18, j) b_2(18, j)- b_2(18, j)^6$} \\ \hline 
$a_1(18, j)$ & \multicolumn{2}{|c|}{$- 2^{\frac{2}{9}} 3^{- \frac{1}{3}}$} & \multicolumn{2}{|c|}{$b_2(18, j)^2$} \\ \hline 
$a_2(18, j)$ & \multicolumn{2}{|c|}{$2^{- \frac{2}{9}} 3^{- \frac{2}{3}}$} & \multicolumn{2}{|c|}{$0$} \\ \hline 
\end{tabular}
\begin{tabular}{|c|c|c|}
\hline
$j$ & $1, \ 11$ & $5, \ 13$ \\ \hline 
$b_0(18, j)$ & $1$ & a root of $z^3 - 9 z^2 - 9 z + 9 = 0$  \\ \hline 
$b_1(18, j)$ & $0$ & a root of $8919936 - 8011872 b_0(18, j) - 9735552 b_0(18, j)^2 + z^9 = 0$  \\ \hline 
$b_2(18, j)$ & $0$ & $\frac{1}{36} b_1(18, j)^2(-45 - 15 b_0(18, j) + 2 b_0(18, j)^2)$  \\ \hline 
$b_3(18, j)$ & $- 2^{- \frac{1}{3}} 3^{-1}$ & $0$ \\ \hline \hline
$j$ & \multicolumn{2}{|c|}{$7,\ 17$} \\ \hline
$a_0(18, j)$ & \multicolumn{2}{|c|}{$-4^{\frac{1}{3}}$} \\ \hline 
$a_1(18, j)$ & \multicolumn{2}{|c|}{a root of $4 + 3 a_0(18, j)^2 z^3 + 6 a_0(18, j) z^6 + z^9 = 0$} \\ \hline 
$a_2(18, j)$ & \multicolumn{2}{|c|}{$- \frac{1}{12} a_1(18, j)^2 (5 a_0(18, j)^2 + 6 a_0(18, j) a_1(18, j) ^3 + a_1(18, j)^6)$} \\ \hline 
$b_0(18, j)$ & \multicolumn{2}{|c|}{$\frac{(12 a_1(18, j) a_2(18, j)^4 - 4 a_2(18, j)^3) a_0(18, j) - a_1(18, j)^3 a_2(18, j)^3 
+ 3 a_1(18, j)^2 a_2(18, j)^2 + 5 a_1(18, j) a_2(18, j) + 1}{16 a_2(18, j)^{9/2}}$} \\ \hline 
$b_1(18, j)$ & \multicolumn{2}{|c|}{$\frac{12 a_0(18, j) a_2(18, j)^3 + 3a_1(18, j)^2 a_2(18, j)^2 - 2 a_1(18, j) a_2(18, j) - 1}{8 a_2(18, j)^{5/2}}$} \\ \hline 
$b_2(18, j)$ & \multicolumn{2}{|c|}{$\frac{3 a_1(18, j) a_2(18, j)^2 + a_2(18, j)}{2 a_2(18, j) ^{3/2}}$} \\ \hline 
$b_3(18, j)$ & \multicolumn{2}{|c|}{$a_2(18, j)^{3/2}$} \\ \hline 
\hline
$j$ & \multicolumn{2}{|c|}{$1, \ 3, \ 7, \ 9$} \\ \hline \hline 
$a_0(20, j)$ & \multicolumn{2}{|c|}{$\frac{1}{80} a_1(20, j)^2 a_2(20, j) (172 - 220 a_2(20, j)^2 + 117 a_1(20, j)^5 + 205 a_1(20, j)^5  a_2(20, j)^2)$} \\ \hline 
$a_1(20, j)$ & \multicolumn{2}{|c|}{a root of $56 - 328 a_2(20, j)^2 - 2500 a_2(20, j)^2 z^5 + 625 z^{10} = 0$} \\ \hline 
$a_2(20, j)$ & \multicolumn{2}{|c|}{a root of $5 z^4 + 2 z^2 + 1 = 0$} \\ \hline \hline 
$j$ & \multicolumn{2}{|c|}{$11, \ 13, \ 17, \ 19$} \\ \hline \hline
$a_0(20, j)$ & \multicolumn{2}{|c|}{$- \frac{1}{16} a_1(20, j)^2 a_2(20, j) (-191 - 188 a_2(20, j)^2 + 69 a_2(20, j)^4 + 6 a_2(20, j)^6 )$} \\ \hline 
$a_1(20, j)$ & \multicolumn{2}{|c|}{a root of $z^5 + 1 - 45 a_2(20, j)^2 - 15 a_2(20, j)^4 + 15 a_2(20, j)^6 = 0$} \\ \hline 
$a_2(20, j)$ & \multicolumn{2}{|c|}{a root of $z^8 + 12 z^6 - 26 z^4 - 52 z^2 + 1 = 0$} \\ \hline \hline 
$b_0(20, j)$ & \multicolumn{2}{|c|}{$a_0(20, j)^{3/2}$} \\ \hline 
$b_1(20, j)$ & \multicolumn{2}{|c|}{$\frac{3 a_0(20, j)^2 a_1(20, j) + 1}{a_0(20, j)^{3/2}}$} \\ \hline 
$b_2(20, j)$ & \multicolumn{2}{|c|}{$\frac{12 a_0(20, j)^5 a_2(20, j) + 3 a_0(20, j)^4 a_1(20, j)^2 - 6 a_0(20, j)^2 a_1(20, j) - 1}{8 a_0(20, j)^{9/2}}$} \\ \hline 
$b_3(20, j)$ & \multicolumn{2}{|c|}{$\frac{12(a_0(20, j)^7 a_1(20, j) - a_0(20, j)^5) a_2(20, j) - a_0(20, j)^6 a_1(20, j)^3 + 15 a_0(20, j)^4 a_1(20, j)^2 
+ 9 a_0(20, j) ^2 a_1(20, j) + 1}{16 a_0(20, j)^{15/2}}$} \\ \hline 
\end{tabular}
\end{center} }
\end{landscape}
\hspace{-0.5cm}and the set of complex numbers $(a _1 , \ \cdots , \ a _5 , \ b _1, \ \cdots , \ b_7)$ 
are solutions of a system of equations 
\begin{eqnarray*}
b_7^2 &=& a_5^3, \\
2 b_6 b_7 &=& 3 a_4 a_5^2+a_5, \\
2 b_5 b_7 + b_6^2 &=& 3 a_3 a_5^2+3 a_4^2 a_5+a_4, \\
2 b_4 b_7 + 2 b_5 b_6 &=& 3 a_2 a_5^2+6 a_3 a_4 a_5+a_4^3+a_3, \\
2 b_3 b_7 + 2 b_4 b_6 + b_5^2 &=& 3 a_1 a_5^2+(6 a_2 a_4+3 a_3^2) a_5+3 a_3 a_4^2+a_2, \\
2 b_2 b_7 + 2 b_3 b_6 + 2 b_4 b_5 &=& (6 a_1 a_4+6 a_2 a_3) a_5+3 a_2 a_4^2+3 a_3^2 a_4+a_1, \\
2 b_1 b_7 + 2 b_2 b_6 + 2 b_3 b_5 + b_4^2 &=& (6 a_1 a_3+3 a_2^2) a_5+3 a_1 a_4^2+6 a_2 a_3 a_4+a_3^3, \\
2 b_1 b_6 + 2 b_2 b_5 + 2 b_3 b_4 &=& 6 a_1 a_2 a_5+(6 a_1 a_3+3 a_2^2) a_4+3 a_2 a_3^2, \\
2 b_1 b_5 + 2 b_2 b_4 + b_3^2 &=& 3 a_1^2 a_5+6 a_1 a_2 a_4+3 a_1 a_3^2+3 a_2^2 a_3, \\
2 b_1 b_4 + 2 b_2 b_3 &=& 3 a_1^2 a_4+6 a_1 a_2 a_3+a_2^3, \\
2 b_1 b_3 + b_2^2 &=& 3 a_1^2 a_3+3 a_1 a_2^2, \\
2 b_1 b_2 &=& 3 a_1^2 a_2, \\
b_1^2 &=& a_1^3+1 . 
\end{eqnarray*}
The system is given by comparing the coefficients of $(b _1 t^{21} + b_2 t^{23} + \cdots + b _7 t^{33})^2$ 
with the coefficients of $(a _1 t^{14} + a _2 t^{16} + \cdots + a _5 t^{22})^3 + t^{42} (a _1 t^{14} + a _2 t^{16} + \cdots + a _5 t^{22}) + t^{42}$. 

\end{Def}

\begin{Th} \label{生成元1}
Let $f : {\cal E}_n \rightarrow {\mathbb P}_{\C}^1 \ (n \in {\mathbb N})$ be the elliptic surfaces 
associated to the elliptic curves $E_n : y^2 = x^3 + t^n x + t^n \ (n \in {\mathbb N})$ over ${\mathbb P}_{\C}^1$. 
Then, for each $n \in {\mathbb N}$, 
${\rm NS} ({\cal E}_n)$ 
has a $\mathbb Q$-basis  
$C_0, \ \infty , \ D_{d, \ j}, \ F_{t, \ a} \ \ (d \in {\rm Adm}_1, \ d \mid n , \ j \in (\Z/ d \Z)^{\times}, \ t \in \Sigma ({\cal E}_n), 1\leq a\leq m_t - 1). $
Moreover if ${\cal E}_n$ is rational {\rm (}i.e., $n = 1, \ 2, \ 3, \ 4, \ 6, \ 7, \ 8$ or $12${\rm )}, then these divisors form a $\Z$-basis. 
\end{Th}
The argument of the proof is the same as the previous one, though the computation is very complicated. 

\subsection{Example 2 in \cite{1}}
Example $2$ in \cite{1} is the minimal elliptic surface whose generic fiber is the elliptic curve defined by 
$$
Y^2 = 4 X^3 - 3 u^n X - u^{2 n} \ \ (u \in \mathbb{P}^1 _{\C}, \ n \in {\mathbb N}) 
$$
over $\C (u)$. 
By putting $ X = -9x/4, \ Y = 27 \sqrt{-1 } y / 4 , \ u = \sqrt[n]{ - 27 / 4} t $, 
the defining equation becomes 
\begin{equation} \label{Ex2}
y^2 = x^3 + t^n x + t^{2 n} \ \ (t \in {\mathbb P}^1_{\C}). 
\end{equation}
We denote by $E^2_n$ the elliptic curve defined by (\ref{Ex2}) and by $f : {\cal E}^2_n \longrightarrow {\mathbb P}_{\C} ^1$ 
the associated elliptic surface. 
Similarly to Example $1$, the Mordell-Weil rank $r = {\rm rank} (E_n (\C (t)))$ is given by 
\begin{equation*}
r = 
\sum_{
d \mid n, \ 
d \in {\rm Adm}_2 
}
\varphi (d) ,
\end{equation*}
where $\varphi $ is the Euler function and ${\rm Adm}_2 = \{ 1, \ 2, \ 5, \ 6, \ 8, \ 9, \ 12, \ 14, \ 20, \ 21, \ 30 \}$. 

We now denote by ${\cal E}^1_{n}$ 
the elliptic surface of Example $1$ 
and by $E^1_n$ 
the associated elliptic curve. 
Recall that 
$$E^1_n : y^2 = x^3 + t^n x + t^n .$$
We assume that $n$ is even and write $n = 2m$. 
By putting $\overline{x} = x/ t^{2 m} , \ \overline{y} = y / t^{3 m} , \ \overline{t} = 1 / t $, 
we obtain 
$$\overline{y}^2 = \overline{x} ^3 + \overline{t}^{2m} \overline{x} + \overline{t}^{4 m} .$$
Since this equation is 
the defining equation of $E_{2 m}^2$, 
we obtain an isomorphism 
\begin{equation} \label{1isom2}
{\cal E}^1_{2 m} \overset{\sim}{\longrightarrow} {\cal E}^2_{2 m}, \ \ (x, \ y, \ t) \longmapsto \left( \frac{x}{t^{2 m}}, \ \frac{y}{t^{3 m}}, \ \frac{1}{t} \right). 
\end{equation}
Therefore we define $\varphi (d)$ rational points of $E^2_d$ for each $d (\geq 2) \in {\rm Adm}_2$ as the images of the points of Definition \ref{sections 1} via this isomorphism. 

\begin{Def} \label{sections 2}
For $d \in {\rm Adm}_2$ and $j \in (\Z / d \Z) ^{\times}$, we define $\C (t)$-rational points $P_{d, \ j}$ of $E^2_{d}$ as follows. 
\begin{eqnarray*}
&& P_{1, \ 1} = \left( 0, \ - t \right) , 
\\
&& P_{2, \ 1} = \left( \sqrt{-1} t , \ - t^{2} \right) , \\
&& P_{5, \ j} = \left( 2^{\frac{2}{5}} \zeta _5 ^{3 j} t^{3} , \ - 2^{\frac{1}{5}} \zeta _5 ^{4 j} t^{4} - t^5 \right) , 
\\
&& P_{6, \ j} = \left( - \zeta _6 ^{4 j} t^{4} , \ \sqrt{-1} \zeta _6 ^{5 j} t^{5} \right) , 
\\
&& P_{8, \ j} = \big( a_2(8,j) t^{4} + a_1(8,j) t^{5} + a_0(8,j) t^{6} , 
\\ && \hspace{3.2cm} 
b_3(8,j) t^{6} + b_2(8,j) t^{7} + b_1(8,j) t^{8} + b_0(8,j) t^{9} \big) ,
\\
&& P_{9, \ j} = \big( a_2(18, \tilde{j}) \zeta _{9} ^{4 j}  t^{4} + a_1(18, \tilde{j}) \zeta _9 ^{5 j} t^{5} + a_0(18, \tilde{j}) \zeta _9 ^{6 j} t^6, \\ 
&& \hspace{2cm} b_3(18, \tilde{j}) \zeta _{9} ^{6 j} t^{6} + b_2(18, \tilde{j}) \zeta _{9} ^{7 j} t^{7} \\
&& \hspace{3cm}+ b_1(18, \tilde{j}) \zeta _{9} ^{8 j} t^{8} + b_0(18, \tilde{j}) t^9 \big) , 
\\
&& P_{12, \ j} = 
\big( a_2(12,j) t^{6} + a_1(12,j) t^{7} + a_0(12,j) t^{8} , \\ 
&& \hspace{1.5cm} b_3(12,j) t^{9} + b_2(12,j) t^{10} + b_1(12,j) t^{11} + b_0(12,j) t^{12} \big) ,
\\
&& P_{14, \ j} = 
\left( - \zeta _{14} ^{8 j} t^{8} -  \zeta _{14} ^{10 j} t^{10} , 
\sqrt{-1} \left( \zeta _{14} ^{11 j} t^{11} + \zeta _{14} ^{13 j} t^{13} + \zeta _{14} ^{15 j} t^{15} \right) \right) , 
\\
&& P_{20, \ j} = \Big( a_2(20,j) \zeta _{20} ^{10 j} t^{10} + a_1(20,j) \zeta _{20} ^{12 j} t^{12} + a_0(20,j) \zeta _{20} ^{14 j} t^{14} , \\ 
&& \hspace{2cm} b_3(20,j) \zeta _{20} ^{15 j } t^{15} + b_2(20,j) \zeta _{20} ^{ 17 j } t^{17} \\
&& \hspace{3cm} + b_1(20,j) \zeta _{20} ^{ 19 j } t^{19} + b_0(20,j) \zeta _{20} ^{21 j } t^{21} \Big) , 
\\
&& P_{21, \ j} = \Big( a_5 \zeta _{21} ^{10 j} t^{10} + a_4 \zeta _{21} ^{11 j} t^{11} + a_3 \zeta _{21} ^{12 j} t^{12} 
+ a_2 \zeta _{21} ^{13 j} t^{13} + a_1 \zeta _{21} ^{14 j} t^{14} , \\
&& \hspace{3cm}b_7 \zeta _{21} ^{15 j} t^{15} + b_6 \zeta _{21} ^{16 j} t^{16} + b_5 \zeta _{21} ^{17 j} t^{17} \\
&& \hspace{4cm}+ b_4 \zeta _{21} ^{18 j} t^{18} + b_3 \zeta _{21} ^{19 j} t^{19} + b_2 \zeta _{21} ^{20 j} t^{20} + b_1 t^{21} \Big) , 
\\
&& P_{30, \ j} = 
\Big( - 3^{\frac{2}{5}} \zeta _{30} ^{16 j} t^{16} - 3^{\frac{1}{5}} \zeta _{30} ^{18 j} t^{18} - \zeta _{30} ^{20 j} t^{20} , \\ 
&& \hspace{2cm}  \sqrt{-1} \left( 3^{\frac{1}{5}} \zeta _{30} ^{23 j} t^{23} +  2 \zeta _{30} ^{25 j} t^{25} 
+ 3^{\frac{4}{5}}  \zeta _{30} ^{27 j} t^{27} +  3^{\frac{1}{5}} \zeta _{30} ^{29 j} t^{29} \right) \Big) , \\
\end{eqnarray*}
where $\tilde{j}$ equals $j$ if $j$ is odd and equals $j + 9$ if $j$ is even,  
and the coefficients $a_k (d,j) , \ b_k (d,j), \ a_1, \ \cdots, \ a_5, \ b_1, \ \cdots , \ b_7$ 
are same as them of 
Definition \ref{sections 1}. 
\end{Def}

The following theorem follows 
from the isomorphism (\ref{1isom2}) and Theorem \ref{生成元1}. 

\begin{Th} \label{生成元2}
Let $f : {\cal E}_n \rightarrow {\mathbb P}_{\C}^1 \ (n \in {\mathbb N})$ be the elliptic surfaces 
associated to the elliptic curves $E_n : y^2 = x^3 + t^n x + t^{2 n} \ (n \in {\mathbb N})$ over ${\mathbb P}_{\C}^1$. 
Then, for each $n \in {\mathbb N}$, ${\rm NS} ({\cal E}_n)$ 
has a $\mathbb Q$-basis 
$C_0, \ \infty , \ D_{d, \ j}, \ F_{t, \ a} \ \ (d \in {\rm Adm}_2, \ d \mid n , \ j \in (\Z/ d \Z)^{\times}, \ t \in \Sigma ({\cal E}_n), 1\leq a\leq m_t - 1). $
Moreover if ${\cal E}_n$ is rational {\rm (}i.e., $n = 1, \ 2, \ 3, \ 4, \ 5, \ 6, \ 8, \ 9$ or $12${\rm )}, then these divisors form a $\Z$-basis. 
\end{Th}

\subsection{Example 3 in \cite{1}}
Example $3$ in \cite{1} is the minimal elliptic surface whose generic fiber is the elliptic curve defined by 
$$
Y^2 = 4 X^3 - 3 u ^{3 n} \left( u^{n} - \frac{8}{9} \right) X + u^{4 n} \left( u^{2 n} - \frac{4}{3} u^n + \frac{8}{27} \right) \ \ (u \in \mathbb{P}^1 _{\C}, \ n \in {\mathbb N}) 
$$
over $\C (u)$. 
By changing the variables suitably, 
the defining equation becomes 
\begin{equation} \label{Ex3}
y^2 = x^3 + x^2 + t^n x + \frac{t^{2 n}}{4} \ \ (t \in {\mathbb P}^1_{\C}). 
\end{equation}
We denote by $E_n$ the elliptic curve defined by (\ref{Ex3}) and by $f : {\cal E}_n \longrightarrow {\mathbb P}_{\C} ^1$ 
the associated elliptic surface. 
By Stiller's list (Example $3$ in \cite{1}), we have that the Mordell-Weil rank $r = {\rm rank} (E_n (\C (t)))$ 
equals $1$ if $n$ is even and equals $0$ if $n$ is odd. 
Therefore 
similarly to Example $1, \ 2$ and $4$, 
we obtain $r = \sum_{d \in {\rm Adm}_3, \ d \mid n} \varphi (d), \ {\rm Adm}_3 = \{ 2 \}$, and  
we can show the following theorem. 

\begin{Th} \label{生成元3}
Let $f : {\cal E}_n \rightarrow {\mathbb P}_{\C}^1 \ (n \in {\mathbb N})$ be the elliptic surfaces 
associated to the elliptic curves $E_n : y^2 = x^3 + x^2 + t^n x + t^{2n} / 4 \ (n \in {\mathbb N})$ over ${\mathbb P}_{\C}^1$. 
For each $n \in {\mathbb N}$, 
if $n$ is odd, 
then ${\rm NS} ({\cal E}_n)$ 
has a $\mathbb Z$-basis 
$C_0, \ \infty , \ F_{t, \ a} \ \ (t \in \Sigma ({\cal E}_n), 1\leq a\leq m_t - 1) $, 
and if $n$ is even, then 
the group has a $\mathbb Q$-basis 
$C_0, \ \infty , \ D_{2, \ 1}, \ F_{t, \ a} \ \ (t \in \Sigma ({\cal E}_n), 1\leq a\leq m_t - 1) $,  
where $D_{2, \ 1} = (P_{2, \ 1}) - \infty$ and $\C (t)$-rational point $P_{2, \ 1}$ is defined by 
$$P_{2, \ 1} = \left( - \frac{1}{2} t^{n } , \ \frac{\sqrt{-2}}{4} t^{\frac{3 n}{2}} \right) .$$
Moreover if ${\cal E}_n$ is rational {\rm (}i.e., $n \leq  3${\rm )}, then these divisors form a $\Z$-basis. 
\end{Th}

\subsection{Example 5 in \cite{1}}
Example $5$ in \cite{1} is the minimal elliptic surface whose generic fiber is the elliptic curve defined by an equation\footnote{The equation in \cite[p.188]{1} is incorrect. } 
\begin{equation*} \label{Ex5 1}
Y^2 = 4 X^3 - 3 u ^{12 k + 3} \left( u^{4 k + 1} - \frac{3}{4} \right) X - u^{20 k + 5} \left( u^{4 k + 1} - \frac{9}{8} \right)  \ \ (u \in \mathbb{P}^1 _{\C}, \ k \in {\mathbb N})
\end{equation*}
over $\C (u)$. 
By changing the variables suitably, 
the defining equation becomes 
\begin{equation*} 
y^2 = x^3 + x^2 + t^{4 k + 1} x \ \ (t \in {\mathbb P}_{\C} ^1).
\end{equation*}

Here we discuss a slightly more general equation 
\begin{equation} \label{Ex5 2}
y^2 = x^3 + x^2 + t^{n} x \ \ (t \in {\mathbb P}_{\C} ^1, \ n \in {\mathbb N}).
\end{equation}
We denote by $E_n$ the elliptic curve defined by (\ref{Ex5 2}) and by $f : {\cal E}_n \longrightarrow {\mathbb P}_{\C} ^1$ 
the associated elliptic surface. 
The surface ${\cal E}_n$ has singular fibers of type ${\rm I}_{2 n}$ over $0$, type ${\rm I}_1$ over $\zeta _n^i \sqrt[n]{1/4} \ (0 \leq i \leq n - 1)$ 
and ${\rm III}^*$ (resp. ${\rm I}_0^*, \ {\rm III}, \ {\rm I}_0$) over $\infty$ as $n \equiv 1$ (resp. $2, \ 3, \ 0$) modulo 4. 
By using Stiller's method, one can show that the Mordell-Weil rank $r = {\rm rank} (E_n (\C (t)))$ is given by 
\begin{equation*} 
r = 
\sum_{
d \mid n, \ 
 d \in {\rm Adm}_5 
}
\varphi (d), 
\end{equation*}
where 
$\varphi $ is the Euler function and ${\rm Adm}_5 = \{ 2, \ 3 \}$. 
We obtain the following theorem similarly to 
the other examples. 


\begin{Th} \label{生成元5}
Let $f : {\cal E}_n \rightarrow {\mathbb P}_{\C}^1 \ (n \in {\mathbb N})$ be the elliptic surfaces 
associated to the elliptic curves $E_n : y^2 = x^3 + x^2 + t^{n} x  \ (n \in {\mathbb N})$ over ${\mathbb P}_{\C}^1$. 
Then, for each $n \in {\mathbb N}$, 
${\rm NS} ({\cal E}_n)$ has a ${\mathbb Q}$-basis $C_0, \ \infty, \ D_{d, \ j}, \ F_{t, \ a} \ 
(d \in {\rm Adm}_5, \ d \mid n, \ j \in (\Z / d \Z)^{\times}, \ t \in \Sigma ({\cal E}_n), \ 1 \leq a \leq m_t - 1)$, 
where 
$D_{d, \ j} = (P_{d, \ j}) - \infty$ and $\C (t)$-rational points $P_{2, \ 1}, \ P_{3, \ j}$ 
are defined by 
\begin{eqnarray*}
&&P_{2, \ 1} = \left( \sqrt{-1} t^{\frac{n}{2}} , \ \sqrt{-1} t^{\frac{n}{2}} \right) , \\
&&P_{3, \ j} = \left( 2^{\frac{2}{3}} \zeta _3^j t^{\frac{n}{3}} , \ 2^{\frac{2}{3}} \zeta _3^j t^{\frac{n}{3}} + 2^{\frac{1}{3}} \zeta _3^{2 j} t^{\frac{2 n}{3}} \right) .
\end{eqnarray*}
Moreover if ${\cal E}_n$ is rational {\rm (}i.e., $n \leq  4${\rm )}, then these divisors form a $\Z$-basis. 
\end{Th}

\section{Alternative proof of Stiller's computations}


Each surface in the Examples $1$, $2$, $3$, $4$ and $5$ has an automorphism such that 
it acts in multiplicity one (i.e. each eigenspace is one-dimensional) 
on the second de Rham cohomology modulo zero and fibral divisor classes. 
Stiller showed this by using the {\it inhomogeneous} de Rham cohomology, and 
this is essentially used in his argument on the computation of Picard numbers.

We gave the explicit $\mathbb Q$-bases 
of the N\'eron-Severi groups in the last section, 
where we used his results on the Picard numbers. 
However once one has the divisors as in the last section, 
one can conclude that they automatically 
form a $\mathbb Q$-basis of 
the N\'eron-Severi group. 
We show it in this section.

\medskip

Let $f : {\cal E}_n \longrightarrow \mathbb{P}_{\C}^1\ (n \in {\mathbb N})$ be one of the families of elliptic surfaces of 
Example $1, \ 2, \ 3, \ 4$ and $5$. 
Let ${\rm NS} ({\cal E}_n) '$ be the subgroup of ${\rm NS} ({\cal E}_n) $ which is generated by all the divisors in Theorems \ref{生成元} etc. 
Put 
$H^2 _{{\rm tr}} ({\cal E}_n) = H^2 ({\cal E}_n , \ \Q ) / {\rm NS} ({\cal E}_n)_{\Q} $, $ V ({\cal E}_n) = H^2 ({\cal E}_n , \ \Q ) / {\rm NS} ({\cal E}_n) ' _{\Q}$. 
The goal is to show ${\rm NS} ({\cal E}_n)_{\Q} = {\rm NS} ({\cal E}_n) ' _{\Q}$, or equivalently
\begin{equation} \label{want}
{\rm dim} \ V ({\cal E}_n) \leq {\rm dim} \ H^2 _{{\rm tr}} ({\cal E}_n) .
\end{equation}

We give a proof of \eqref{want} only for Example $1$ since the same argument works in the other cases.
We already know the dimension of $V ({\cal E}_n)$. In the case of Example $1$, the result can be written as follows. 
\begin{equation}\label{dim1}
{\rm dim} \ V ({\cal E}_n) = \sum_{d \in S_n^1} \varphi (d) , 
\end{equation}
where we put $S_n^1 = \{ d \in {\mathbb N} \ ; \ d \mid n , \ d \not \in {\rm Adm}_1 \cup \{ 4, \ 6 \} \}$ and $\varphi (d)$ is the Euler function. 
In particular, when $n = 1, \ 2, \ 3, \ 4, \ 6, \ 7, \ 8$ or $12$, then \eqref{dim1} is zero, so that there is nothing to prove. 
We assume $n \not = 1, \ 2, \ 3, \ 4, \ 6, \ 7, \ 8$ or $12$.

Let $\sigma : {\cal E}_n \longrightarrow {\cal E}_n$ be an automorphism which is defined by $(x, \ y , \ t) \mapsto (x, \ y , \ \zeta _n ^{-1} t)$, 
and let $\sigma ^*$ be the automorphism on $H^2_{{\rm tr}} ({\cal E}_n) 
$ induced by $\sigma $. 
We denote by $f(T)$ the minimal polynomial of $\sigma ^*$ over $\mathbb Q$. 
If we have 
\begin{equation} \label{last}
f (\zeta _d) = 0 \ \  {\rm for \ each} \ \ d \in S_n ^1,
\end{equation}
then $d$-th cyclotomic polynomial divides into $f(T)$ and hence we have 
$${\rm dim} \ {H}^2 _{{\rm tr}} ({\cal E}_n) 
\geq {\rm deg} \ f(T) \geq 
\sum_{ d \in S_n^1
}
\varphi (d)
=  {\rm dim} \ V ({\cal E}_n)$$
and \eqref{want} follows. 
Let us prove \eqref{last}. 
\begin{Lem} 
Let $n = 12 l + k$ with $l \geq 0 , \ 1 \leq k \leq 12$. 
If $n$ equals $1$, $2$, $3$, $4$, $6$, $7$, $8$ or $12$, then $H^0 ({\cal E}_n , \ \Omega _{{\cal E}_n} ^2) = 0. $ 
Otherwise, 
$H^0 ({\cal E}_n , \ \Omega _{{\cal E}_n} ^2)$ has a basis 
\begin{eqnarray*}
t^{2 n - a(n) - 3} d t \frac{d x}{y}, \ \cdots , \ t^{2 n - a(n) - b(n) - 3} d t \frac{d x}{y}, 
\end{eqnarray*}
where $a(n), \ b(n)$ are defined as follows: 
\begin{center}
\begin{tabular}{|c|c|c|c|c|c|c|}
\hline
$k$ & $1$ & $2$ & $3$ & $4$ & $5$ & $6$ \\ \hline 
$a(n)$ & $9l - 1$ & $9l$ & $9l + 1$ & $9l +2$ & $9l +2$ & $9l +3$ \\ \hline 
$b(n)$ & $l - 1$ & $l - 1$ & $l - 1$ & $l - 1$ & $l$ & $l - 1$ \\ 
\hline \hline
$k$ & $7$ & $8$ & $9$ & $10$ & $11$ & $12$ \\ \hline 
$a(n)$ & $9l +4$ & $9l +5$ & $9l +5$ & $9l +6$ & $9l +7$ & $9l +8$ \\ \hline
$b(n)$ & $l - 1$ & $l - 1$ & $l$ & $l$ & $l$ & $l - 1$ \\ \hline
\end{tabular}
\end{center}
\end{Lem}

\begin{proof}
Exercise (e.g. \cite{1} Prop. $3.3$ for details).
\end{proof}

For an integer $i$ with $0 \leq i \leq b(n)$, 
since we have 
\begin{eqnarray*}
\sigma ^* \left( t^{2 n - a(n) - i - 3} d t \frac{d x}{y} \right) = \zeta  _n ^{a(n) + i + 2} \left( t^{2 n - a(n) - i - 3} d t \frac{d x}{y} \right) ,
\end{eqnarray*}
the automorphism $\sigma ^*$ on $H_{\rm tr}^2 ({\cal E}_n)$ over ${\mathbb Q}$ has eigenvalues $\zeta _n ^{a(n) + i + 2}$, 
so we have $f(\zeta _n ^{a(n) + i + 2}) = 0 \ (0 \leq i \leq b (n))$. 
On the other hand, since we have 
$$J (n) := \left\{ a (n) + 2, \ \cdots , \ a(n) + b(n) + 2 \right\} = \{ j \in {\mathbb N} \mid 9 n < 12 j < 10 n \} ,$$
we obtain 
$\left\{ (a (d) + 2) n/ d, \ \cdots , \ ( a(d) + b(d) + 2 ) n /d \right\} \subset J (n)$ 
for each $d$ which divides into $n$. 
Then $J (d) = \emptyset$ if and only if $d = 1, \ 2, \ 3, \ 4, \ 6, \ 7, \ 8$ or $12$. 
In addition, $d \in {\rm Adm}_1$ if and only if 
each $j \in \{ j \in {\mathbb N} \mid 9 d \leq 12 j \leq 10 d \}$ 
is not relatively prime to $d$ (see Remark \ref{Rem}). 
Therefore for each $d \in S_n^1$, there exists a natural number $j$ which is relatively prime to $d$ such that $j n /d \in J (n)$, and 
we have $\zeta _n ^{j n /d} = \zeta _d ^j $. 
This implies \eqref{last}. 

\section{Acknowledgements}

The author would like to thank Professor Masanori Asakura for helpful comments and suggestions. 
He 
also thanks Professor Matthias Sch$\ddot{\rm u}$tt whose comments made enormous contribution to this paper.

\smallskip
Masamichi Kuroda \\
Department of Mathematics \\
Hokkaido University \\
Sapporo 060-0810 \\
Japan 

\smallskip
\hspace{-0.5cm}m-kuroda@math.sci.hokudai.ac.jp

\end{document}